\documentclass[english, 11pt]{amsart}

\usepackage[english]{babel}
\usepackage[usenames,dvipsnames,svgnames,table]{xcolor}
\usepackage{amsfonts}
\usepackage{enumerate}
\usepackage{amsthm,amsmath,amssymb}
\usepackage{lineno}
\usepackage{ulem}

\usepackage[margin=3cm]{geometry}
\newtheorem{theorem}{Theorem}[section]
\newtheorem*{theorem*}{Theorem}

\newtheorem{proposition}[theorem]{Proposition}
\newtheorem*{prop*}{Proposition}

\newtheorem*{corollary*}{Corollary}
\newtheorem{corollary}[theorem]{Corollary}

\newtheorem{lemma}[theorem]{Lemma}
\newtheorem{claim}[theorem]{Claim}

\newtheorem{conjecture}{Conjecture}
\newtheorem{def-theorem}[theorem]{Theorem-Definition}
\newtheorem*{statement*}{Statement}

\theoremstyle{definition}
\newtheorem{remark}[theorem]{Remark}
\newtheorem*{thm*}{Theorem}
\newtheorem{question}{Question}
\newtheorem{example}[theorem]{Example}

\newtheorem{definition}[theorem]{Definition}

\renewcommand{\subset}{\subseteq}

\newcommand{\Q}{\mathbb{Q}}

\newcommand{\N}{\mathbb{N}}
\newcommand{\NN}{\mathbb{N}}

\newcommand{\R}{\mathbb{R}}

\newcommand{\C}{\mathbb{C}}

\newcommand{\Lm}{\mathcal{L}}
\newcommand{\cL}{\Lm}

\newcommand{\Lring}{\Lm_{\text{ring}}}

\newcommand{\RV}{\text{RV}}
\newcommand{\rv}{\text{rv}}

\title[Definable functions in tame expansions of ACVF]{Definable functions in tame expansions of algebraically closed valued fields}
\author{Pablo Cubides Kovacsics }
\thanks{The first author was funded by the European Research Council, ERC Grant nr. 615722, TOSSIBERG and partially funded by ERC Grant nr. 615722, MOTMELSUM. The second author was partially supported by ValCoMo, Projet ANR blanc ANR-13-BS01-0006.}
\author{Fran\c coise Delon}

\address{Pablo Cubides Kovacsics, Universit\'e de Caen, Laboratoire de math\'ematiques Nicolas Oresme, CNRS UMR 6139,
14032 Caen cedex, France }
\email{pablo.cubides@unicaen.fr}
\address{Fran\c coise Delon, CNRS, IMJ - PRG, Universit\'e Paris Diderot,
UFR de math\'ematiques, case 7012, 75205 Paris Cedex 13, France.}
\email{delon@math.univ-paris-diderot.fr}

\begin{document}
\begin{abstract} 
In this article we study definable functions in tame expansions of algebraically closed valued fields. For a given definable function we have two types of results: of type (I), which hold at a neighborhood of infinity, and of type (II), which hold locally for all but finitely many points in the domain of the function. In the first part of the article, we show type (I) and (II) results concerning factorizations of definable functions over the value group. As an application, we show that tame expansions of algebraically closed valued fields having value group $\Q$ (like $\C_p$ and $\overline{\mathbb{F}_p}^{\mathrm{alg}}(\!(t^\Q)\!)$) are polynomially bounded. In the second part, under an additional assumption on the asymptotic behavior of unary definable functions of the value group, we extend these factorizations over the residue multiplicative structure $\RV$. In characteristic 0, we obtain as a corollary that the domain of a definable function $f\colon X\subseteq K\to K$ can be partitioned into sets $F\cup E\cup J$, where $F$ is finite, $f|E$ is locally constant and $f|J$ satisfies locally the Jacobian property. 
\end{abstract}

\maketitle
\normalem 

The o-minimality of the real field expanded by the exponential function, proved by Wilkie in \cite{wilkie96}, was one of the major achievements of model theory during the nineties. Its ubiquity and importance can also be seen in the light of Miller's growth dichotomy \cite{miller94}: an o-minimal expansion of a real field is either polynomially bounded or the exponential function is definable in it. 

During the same years, analogous minimality conditions were introduced as candidates to define tame expansions of different structures. In the case of expansions of algebraically closed valued fields, the corresponding notion was introduced by Haskell, Macpherson and Steinhorn in \cite{macphersonETAL:94, macphersonETAL:96} and is called \emph{$C$-minimality}. 

In this article we study definable functions in tame expansions of algebraically closed valued fields where $C$-minimality is taken as the tameness condition. The main theorems obtained concern the factorization of definable functions over the value group ($\Gamma$-factorization) and over the residue multiplicative structure ($\RV$-factorization). The formal definition of what factorization means will be given in the next section. For both factorizations we have two types of results: of type (I), which hold at a neighborhood of infinity, that is, they hold outside some neighborhood of $0$; and of type (II), which hold locally, for all but finitely many points of the domain. 

As an application of type (I) $\Gamma$-factorization, we show a major difference concerning the behavior at infinity of definable functions with respect to the o-minimal context: $C$-minimal valued fields having $\Q$ as their value group are polynomially bounded. This yields, in particular, that $C$-minimal expansions of algebraically closed valued fields like $\C_p$ or $\overline{\mathbb{F}_p}^{\mathrm{alg}}(\!(t^\Q)\!)$ are polynomially bounded, which radically restricts the class of definable functions in such expansions. 

$\RV$-factorizations are proven under an additional hypothesis on the asymptotic behavior of unary definable functions of the value group. In characteristic 0, we deduce from type (II) $\RV$-factorization that the domain of a definable function $f\colon X\subseteq K \to K$ can be partitioned into sets $F\cup E\cup J$, where $F$ is finite, $f|E$ is locally constant and $f|J$ satisfies locally the Jacobian property (as defined in \cite{cluckers-lipshitz:2011}). Let us now introduce these concepts and formally state the results of the article. 

\section{Main results} 

Let $(K,v)$ be a valued field. All valued fields under consideration are non-trivially valued. We denote the value group by $\Gamma_K$, the valuation ring by~$\mathcal{O}_K$, its maximal ideal by~$\mathcal{M}_K$ and the residue field by $K/v$. For $a\in K$ and $\gamma\in \Gamma_K\cup\{+\infty\}$, we let 
$B(a,\gamma):=\{x\in K: v(x-a)\geq \gamma\}$ denote the closed ball centered at $a$ of radius $\gamma$. Respectively, for $\gamma\in \Gamma_K\cup\{-\infty\}$, we let $B^\circ(a,\gamma):=\{x\in K: v(x-a)>\gamma\}$ denote the open ball centered at $a$ of radius $\gamma$ (thus $K$ is treated as an open ball). By a ball we mean either a closed or  an open ball. We let $\RV^\ast$ denote the quotient group 
\[
\RV^\ast:= K^\times/(1 + \mathcal{M}_K) 
\] 
and define $\RV:=\RV^\ast\cup\{0\}$. The function $\rv\colon K\to \RV$ denotes the quotient map which in addition sends $0$ to $0$.

We study $(K,v)$ as a first order structure using the language of valued fields $\Lm_{div}:=(+,-,\cdot,0,1,\text{div})$ where $\text{div}(x,y)$ is a binary predicate interpreted in $(K,v)$ by $v(x)\leq v(y)$. Given a language $\Lm$ extending $\Lm_{div}$, by an $\Lm$-definable set we mean a set defined in the language~$\cL$ with parameters. We will often omit the prefix $\cL$ and talk about definable sets when~$\cL$ is clear from the context. Note that $\Gamma_K$, $\RV$ and $K/v$ are all three $\cL_{div}$-interpretable. Abusing of terminology, by a definable subset of $\Gamma_K$, $K/v$ or $\RV$ we mean an interpretable subset. Let us recall the definition of $C$-minimality in this context: 

\begin{definition}\label{def:tame}
An expansion $(K,\cL)$ of $(K,\cL_{div})$ is \emph{$C$-minimal} if for every elementary equivalent structure $(K',\Lm)$, every $\Lm$-definable subset $X\subseteq K'$ is a boolean combination of balls. 
\end{definition}

Every algebraically closed valued field $(K,\cL_{div})$ is $C$-minimal. Conversely, by a result of Haskell and Macpherson in \cite{macphersonETAL:94}, every $C$-minimal valued field is algebraically closed. Further examples of $C$-minimal valued fields include algebraically closed valued fields with analytic structure as studied by Lipshitz and Robinson in \cite{lipshitzETAL:98}. From now on we work over a $C$-minimal expansion $(K,\cL)$ of an algebraically closed valued field $(K,v)$.  

The first part of the paper (Sections 3 and 4) is devoted to study $\Gamma$-factorizations. Let us provide their formal definition. 

\begin{definition}\label{def:gammafact} Let $f\colon X\subseteq K \rightarrow K$ be a function. 
\begin{enumerate}[(1)]
\item Suppose that $f$ is defined at a neighborhood of infinity. The function $f$ \emph{factorizes at infinity over $\Gamma$} if there is $h\colon \Gamma_K \rightarrow \Gamma_K$ and $\gamma_0\in \Gamma_K$ such that $v(f(x))=h(v(x))$ for all $x\in X\setminus B(0,\gamma_0)$. We say in this case that $f$ factorizes at infinity over $\Gamma$ \emph{through $h$} or that $h$ is a \emph{$\Gamma$-factorization of $f$ at infinity}.  
\item A function $f\colon X\subseteq K \rightarrow K$ \emph{factorizes over $\Gamma$} if there is a function $h\colon \Gamma_K \rightarrow \Gamma_K$
such that $v(f(x)-f(y))=h(v(x-y))$ for all distinct $x,y \in X$. In this case we say that $f$ factorizes over $\Gamma$ \emph{through $h$} or that $h$ is a \emph{$\Gamma$-factorization of $f$}. 
\item We say that $f$ \emph{locally factorizes over $\Gamma$} if for every $x\in X$ there is an open ball $B_x\subseteq X$ containing $x$ such that $f|B_x$ factorizes over $\Gamma$. For $h\colon Y\subseteq X\times \Gamma_K \rightarrow \Gamma_K$, we say that $f$ locally factorizes over $\Gamma$ \emph{through $h$} if for every $x\in X$ there is an open ball $B_x\subseteq X$ containing $x$ such that $f|B_x$ factorizes over $\Gamma$ through $h_x$. 
\end{enumerate}
\end{definition}

The following results correspond to type (I) and (II)  $\Gamma$-factorization (later Theorems \ref{thm:gamma-fact-infty} and \ref{thm:local-gamma-fact}). 

\begin{theorem*}[$\Gamma$-factorization I] Let $(K,\cL)$ be a $C$-minimal valued field and let $f\colon X\subseteq K\to K^\times$ be a definable function defined at a neighborhood of infinity. Then $f$ factorizes at infinity over $\Gamma$ through a definable function $h$.  
\end{theorem*}

\begin{theorem*}[$\Gamma$-factorization II] Let $(K,\cL)$ be a $C$-minimal valued field and let $f\colon X\subseteq K\to K$ be a definable local $C$-isomorphism. Then there is a finite subset $F\subseteq X$ such that $f| (X\setminus F)$ locally factorizes over $\Gamma$ through a definable function $h$.  
\end{theorem*}

It is natural to restrict our study to definable local $C$-isomorphisms (Definition \ref{def:Ciso}) in the second theorem above, since by a result of Haskell and Macpherson (see Theorem \ref{thm:hasmac}), every definable function $f\colon X\subseteq K\to K$, $X$ definably decomposes into $X=F\cup E\cup J$ where $F$ is a finite set, $f|_{E}$ is locally constant and $f|_{J}$ is a local $C$-isomorphism.  

Both type (I) and type (II) $\Gamma$-factorizations also hold uniformly in definable families (see later Theorems \ref{thm:uniform-gamma-fact-infty} and \ref{thm:uniform-local-gamma-fact}). 

From type (I) $\Gamma$-factorization we deduce the above mentioned result about polynomially bounded $C$-minimal valued fields. Let us recall what polynomially bounded means in this context. The structure $(K,\cL)$ is said to be \emph{polynomially bounded} if for every definable function $f\colon X\subseteq K\to K$ there is $\gamma\in \Gamma_K$ and a non-zero integer $n$ such that $v(f(x))>nv(x)$ for all $x\in X\setminus B(0,\gamma)$. We say it is \emph{uniformly polynomially bounded}, if the analogous statement holds over definable families (see Definition \ref{def:polybdd} for the precise definition).

\begin{theorem*}[later Theorem \ref{main}] Let $(K,\cL)$ be a $C$-minimal valued field with $\Gamma_K=\Q$. Then $(K,\cL)$ is uniformly polynomially bounded. In particular, any $C$-minimal expansion of $\C_p$ or $\overline{\mathbb{F}_p}^{\mathrm{alg}}(\!(t^\Q)\!)$ is polynomially bounded. More generally, any $C$-minimal valued field $(K,\cL)$ which is $\cL$-elementary equivalent to a valued field having $\Q$ as its value group, is uniformly polynomially bounded. 
\end{theorem*}

The previous theorem is obtained using a dichotomy for o-minimal expansions of ordered groups due to Miller and Starchenko \cite{miller-starchenko98}, and the following result (which uses type (I) $\Gamma$-factorization).

\begin{theorem*}[later Theorem \ref{thm:germs}] Let $(K,\cL)$ be a $C$-minimal valued field such that $\Gamma_K$ is $\Q$-linearly bounded. Then $(K,\cL)$ is uniformly polynomially bounded.  
\end{theorem*}

In the second part of the paper (Sections 5 and 6), we extend $\Gamma$-factorizations to $\RV$-factorizations under additional hypotheses concerning definable functions of the value group. In particular, we derive $\RV$-factorization under the hypothesis that such functions are \emph{eventually $\Q$-linear} (see later Definition \ref{def:eventually}). For simplicity, we will now state both type (I) and (II) $\RV$-factorization only for definable unary functions, although we will later prove these results in families (see later Theorems \ref{thm:LJPatinfty} and \ref{thm:almostjac}).   

Let $p$ denote the \emph{characteristic exponent} of $(K,v)$, that is, $p=1$ if the characteristic of $K$ is $0$ and otherwise $p$ equals the characteristic of $K$.

\begin{theorem*}[$\RV$-factorization I] Let $(K,\cL)$ be a $C$-minimal valued field, $f\colon X\subseteq K\rightarrow K^\times$ be a definable function defined at a neighborhood of infinity. Let $h\colon Y\subseteq \Gamma_K \rightarrow \Gamma_K$ be a $\Gamma$-factorization of $f$ at infinity. If $h$ is eventually $\Q$-linear, then there are integers $m,n \in \mathbb Z$ and $c\in \RV^*$ such that, in a neighborhood of infinity,  
\[
\rv(f(x))=\rv(x)^{n/p^{m}}c.
\]  
If in addition $(K,\cL)$ is definably complete and all definable unary $\Gamma$-functions are eventually linear, the limit 
\[
a:=\lim_{x\to \infty} \frac{f(x)}{x^{n/p^{m}}}
\]
exists in $K$ and $c=\rv(a)$. 
\end{theorem*}

\begin{theorem*}[$\RV$-factorization II] Let $(K,\cL)$ be a $C$-minimal valued field and $f\colon X\subseteq K\rightarrow K$ be a definable local $C$-isomorphism. Suppose $h\colon Y\subseteq X\times \Gamma_K\to\Gamma_K$ is a local $\Gamma$-factorization of $f$ which is eventually $\Q$-linear. Then there are a finite definable partition $X=X_1\cup\dots\cup X_\ell$, integers $n_1,\ldots,n_\ell$ and definable functions $\delta\colon X\to\Gamma_K$ and $c\colon X \to \RV^*$ such that for all $x\in X_i$, $B^\circ(x,\delta(x))\subseteq X_i$ and for all distinct $y,z\in B^\circ(x,\delta(x))$, 
\[
\rv((f(y)-f(z))=\rv((y-z))^{p^{n_i}} c(x).
\] 
In particular, when $p=1$ we may assume that $\ell=1$. If in addition $(K,\cL)$ is definably complete and all definable unary $\Gamma$-functions are eventually linear, the limit 
\[
a(x):=\lim_{y\to x} \frac{f(x)-f(y)}{(x-y)^{p^{n_i}}}
\]
exists in $K$ and $c(x)=\rv(a(x))$. 
\end{theorem*}

Recall that an expansion $(K,\cL)$ is\emph{ definable complete} if every definable family of nested balls whose set of radii tends to $+\infty$ has non-empty intersection. Clearly, every expansion of a complete valued field is definably complete. Note moreover that definable completeness is a first order property. Therefore, since every algebraically closed valued field $(K,\cL_{div})$ in the language of valued fields has a complete elementary extension, they are all definably complete. Unfortunately, there are $C$-minimal expansions of valued fields which are not definably complete (see for example \cite[Theorem 5.4]{delonCorps:12}). Nonetheless, since most algebraically closed valued fields of interest are complete (for example $\C_p$), any of their $C$-minimal expansions will satisfy this assumption. 

In characteristic 0, type (II) $\RV$-factorization implies the following result for definably complete $C$-minimal valued fields. 

\begin{theorem*}[later Theorem \ref{thm:LJP} in families]
Let $(K,\cL)$ be a definably complete $C$-minimal valued field of characteristic 0 in which all definable unary $\Gamma$-functions are eventually $\Q$-linear. Let $f\colon X\subseteq K\rightarrow K$ be a definable local $C$-isomorphism. Then there is a finite set $F$ such that $f|(X\setminus F)$ has locally the Jacobian property. 
\end{theorem*}

The Jacobian property is taken from \cite{cluckers-lipshitz:2011} and will be later recalled (Definition \ref{def:jacobian}). A weaker version of Theorem \ref{thm:LJP} was already obtained by the second author for valued fields of equi-characteristic zero in \cite{delonCorps:12}. Here we generalize and extend the result to mixed characteristic. We also obtain the following corollaries. 

\begin{corollary*}[later Corollary \ref{cor:Jac} in families]
Let $(K,\cL)$ be as in the previous theorem and $f\colon X\subseteq K\to K$ be a definable function. Then there is a definable partition of $X$ into sets $X=F\cup E\cup J$ where $F$ is finite, $f|_{E}$ is locally constant and $f|_{J}$ has locally the Jacobian property. 
\end{corollary*}

\begin{corollary*}\label{diff0}
Let $(K,\cL)$ and $f$ be as in the previous corollary. Let $D$ be the definable set $D:=\{x\in X: f'(x)=0\}$. Then $D$ can be decomposed into sets $F\cup L$ such that $F$ is finite and $f|L$ is locally constant. In particular, $f$ is $C^1$ on a cofinite subset of $X$. 
\end{corollary*}

Note that the previous corollary cannot be extended to characteristic $p>0$. Indeed, in an algebraically closed valued field of characteristic $p>0$ the function $x\mapsto x^p$ is injective and has null derivative. 

\

It is worthy to note that over complete algebraically closed valued fields, $C$-minimal expansions by analytic functions as studied by Lipschitz and Robinson \cite{lipshitzETAL:98} do satisfy all hypothesis of Theorem \ref{thm:LJP}, but our theorem is not giving anything new. Indeed, much stronger results follow (like local analyticity of definable functions) from \cite{lipshitzETAL:98} and \cite{cluckers-lipshitz:2011}. On the other hand, type (II) $\Gamma$-factorization and $\RV$-factorization generalize results of Hrushovski-Kazhdan in \cite[Section 5]{hrushovskiETAL:05}. Their results hold in algebraically closed valued fields of residue characteristic 0 under the stronger assumption of $V$-minimality, which implies both definable completeness and that the induced structure on $\RV$ is exactly the induced structure by $\cL_{div}$. As far as we know, type (I) factorizations are new in the literature. 


We would also like to point out that, modulo the following conjecture about o-minimal expansions of $(\Q,<,+,0)$, all definable functions would have $\RV$-factorizations in $C$-minimal valued fields $(K,\cL)$ for which $\Gamma_K=\Q$. In particular, if $(K,\cL)$ is definably complete of characteristic 0, the previous corollary implies that every definable function $f\colon X\subseteq K\to K$ is $C^1$, for all but finitely many points in $X$.

\begin{conjecture}\label{conj:Q} Every definable function $f\colon\Q\to\Q$ in an o-minimal expansion of $(\Q,<,+,0)$ is eventually $\Q$-linear. 
\end{conjecture}

The article is organized as follows. In Section 2, we provide the needed background on $C$-minimality. Section 3 is devoted to show the results about $\Gamma$-factorization. Polynomially bounded $C$-minimal valued fields are studied in Section 4. The results on $\RV$-factorizations are presented in Section 5. Finally, all results related to the Jacobian property are shown in Section 6. 

\subsubsection*{Acknowledgments} 
We would like to thank Raf Cluckers for interesting discussions around the Jacobian property.

\section{Preliminaries and auxiliary results}

Hereafter, $(K,v)$ will denote an algebraically closed valued field, $(K,\cL)$ a $C$-minimal expansion of $(K,\cL_{div})$. We define a ternary relation on $K$, called the \emph{$C$-relation}, by 
\[C(x,y,z)\Leftrightarrow v(x-y)=v(x-z)<v(y-z).\]
Boolean combinations of balls correspond precisely to quantifier free formulas in the language which only contains a predicate for the $C$-relation, which partly explains the analogy between o-minimality and $C$-minimality. 

 We use the following conventions regarding definable families of sets and functions. Given a definable set $W$, a definable family of sets in $K$ parametrized by $W$ is a definable set $X\subseteq W\times K$ such that the projection of $X$ onto the coordinates of $W$ is equal to the set $W$. We will often omit the mention ``in $K$ parametrized by $W$'' and simply write $X\subseteq W\times K$ when no confusion arises. Given a definable family $X\subseteq W\times K$, the fiber at $w\in W$ corresponds to $X_w:=\{x\in K: (w,x)\in X\}$. Similarly, a definable family of functions is a definable function $f\colon X\subseteq W\times K\to K$, and we use the notation $f_w\colon X_w\to K$ for the function given by $f_w(x):=f(w,x)$. 

\begin{remark}\label{rmk:swiss} By $C$-minimality, every definable non-empty subset $X\subseteq K$ is a finite disjoint union of ``Swiss cheeses''. A Swiss cheese is a set of the form $B\setminus (\bigcup_{i=1}^m B_i)$ where $B$ is a ball and each $B_i$ is a ball strictly contained in $B$ such that $B_i\cap B_j=\emptyset$ for $i\neq j$. Furthermore, this also holds in families: if $X\subseteq W\times K$ is a definable family, then there is a finite partition $W_1,\ldots,W_n$ of $W$ and integers $k_i, m_i$ with $i\in\{1,\ldots,n\}$ such that for all $w\in W_i$, the fiber $X_w$ is a disjoint union of $k_i$ Swiss cheeses of the form $B(w)\setminus (\bigcup_{i=1}^{m_i} B_{i}(w)$).
\end{remark}

The reader interested in the general study of $C$-minimal structures is referred both to the original articles \cite{macphersonETAL:94, macphersonETAL:96} and to \cite{cubides:14,delon:11, delonCorps:12}. 

The following two theorems due to Haskell and Macpherson in \cite{macphersonETAL:94} will be extensively used.  

\begin{theorem}[Haskell-Macpherson]\label{lem:omin} Let $(K,\cL)$ be $C$-minimal. 
\begin{enumerate}
\item The induced structure on $\Gamma_K$ is o-minimal, that is, every definable set $Y\subseteq \Gamma_K$ is a finite union of intervals and points.  
\item Any definable function $g\colon X\subseteq K\to \Gamma_K$ is locally constant on a cofinite subset of  $X$.
\item The induced structure on $K/v$ is strongly-minimal. In particular, every definable set $Y\subseteq K/v$ is either finite or cofinite, and no infinite subset of $K/v$ can be linearly ordered by a definable relation.   
\item Let $W$ be a definable set and $X\subseteq W\times K$ be a definable family such that each fiber $X_w$ is finite. Then there is $n\in \mathbb{N}$ such that each fiber has cardinality less than $n$.  
 \end{enumerate}
\end{theorem}

To state their second theorem we need to introduce the notion of $C$-isomorphism. 

\begin{definition}\label{def:Ciso} Given an open ball $B\subseteq K$, a function $f\colon B\to K$ is a \emph{$C$-isomorphism} if $f(B)$ is an open ball and $f$ preserves both the $C$-relation and its negation, that is, if for all $x,y,z\in B$
\[C(x,y,z)\Leftrightarrow C(f(x),f(y),f(z)).\]
\end{definition}

Note that a $C$-isomorphism must be injective. The original definition of $C$-isomorphism in \cite{macphersonETAL:94} does not include the condition ``$f(B)$ is an open ball'', but for our purposes such assumption is harmless. Indeed, for $f\colon X\subseteq K\to K$ a definable local $C$-isomorphism as defined in \cite{macphersonETAL:94} (i.e., without the assumption ``$f(B)$ is an open ball''), by $C$-minimality we have that the set 
\[
\{x\in X: \text{there is no ball $B\subseteq X$ containing $x$ such that f(B) is an open ball}\},
\] 
is finite (one can use Fact 1 in \cite{macphersonETAL:94}, Lemma 1.9 in \cite{cubides:14}).

\begin{theorem}[Haskell-Macpherson]\label{thm:hasmac} Let $(K,\cL)$ be $C$-minimal and $f\colon X\subseteq K\to K$ be a definable function. Then there is a definable partition of $X$ into sets $F\cup E\cup I$ such that $F$ is finite, $f|(X\setminus F)$ is continuous, $f|E$ is locally constant and $f|I$ is a local $C$-isomorphism.
\end{theorem}

Using part (4) of Theorem \ref{lem:omin}, it is not difficult to see that Theorem \ref{thm:hasmac} holds also for definable families of functions as every condition is uniformly definable. For a definable set $W$ and a definable family of functions $f\colon X\subseteq W\times K\rightarrow K$, we say that $f$ is a family of local $C$-isomorphisms if for each $w\in W$ the fiber $f_w$ is a local $C$-isomorphism. The following is the family version of Theorem \ref{thm:hasmac}. 

\begin{corollary}\label{cor:hasmac} Let $(K,\cL)$ be $C$-minimal, $W$ be a definable set and $f\colon X\subseteq W\times K\to K$ be a definable family of functions. Then there is a definable partition of $X$ into sets $F\cup E\cup I$ such that for all $w\in W$, $F_w$ is finite, $f|(X_w\setminus F_w)$ is continuous, $f_w|E_w$ is locally constant and $f_w|I_w$ is a local $C$-isomorphism.
\end{corollary}

Note that for all $x,y\in K^\ast$
\begin{equation*}\label{eq:rv}
\rv(x)=\rv(y) \Leftrightarrow v(x)=v(y)<v(x-y),
\end{equation*}
and therefore the condition $C(0,x,y)$ recovers the equivalence relation defined by $rv(x)=rv(y)$. We will often denote this relation, for notation convenience, by $x\sim y$. We will later need the following lemma about $\RV$-classes. 

\begin{lemma}\label{lem:sim}
Suppose that $a,b,c\in K^\times$ are elements lying in different $RV$-classes. If $a\sim a', b \sim b'$ and $c\sim c'$ then $C(a,b,c) \Leftrightarrow C(a',b',c')$.
\end{lemma}
\begin{proof} We show the following two equivalences from which the result follows: 
\begin{enumerate}
\item $C(a,b,c) \Leftrightarrow C(a',b,c)$;
\item $C(a,b,c) \Leftrightarrow C(a,b',c)$.
\end{enumerate}
By assumption on $a,b,c$ and $a'$ we have that
\begin{center}
\begin{tabular}{lll}
$C(a,b,c)$ & $\Leftrightarrow$ & $v(a)<\min(v(b),v(c))$\\
& $\Leftrightarrow$ & $v(a)=v(a')<\min(v(b),v(c))$\\
& $\Leftrightarrow$ & $C(a',b,c)$,
\end{tabular}
\end{center}
which shows (1). For (2),

\begin{center}
\begin{tabular}{lll}
$C(a,b,c)$ & $\Leftrightarrow$ & $v(a)<\min(v(b),v(c))$\\
& $\Leftrightarrow$ & $v(a)<\min(v(b'),v(c))$\\
& $\Leftrightarrow$ & $C(a,b',c)$.
\end{tabular}
\end{center}
\end{proof}

\ 

Let $(M,\cL)$ be a structure and $S$ be a $\emptyset$-definable ($\emptyset$-interpretable) set. The induced structure by $\cL$ on $S$ will be denoted by $(S,\cL_{ind})$. It consists of all $\emptyset$-$\cL$-definable (resp. $\emptyset$-$\cL$-interpretable) subsets of cartesian powers of $S$. A $\emptyset$-definable ($\emptyset$-interpretable) set $S$ is \emph{stably embedded} if every $\cL$-definable subset $X\subseteq S^n$ is already $\cL_{ind}$-definable with parameters in $S$. The following result is a particular case of Corollary 1.10 from \cite{cubidesPHD2013}. It essentially follows from \cite[Theorem 1.4]{pillay2011}. 

\begin{proposition}\label{prop:stably} Let $K$ be a $C$-minimal valued field. Then $\Gamma_K$ and $K/v$ are stably embedded. 
\end{proposition}

In our setting there are morally three infinities. We have $-\infty$ and $+\infty$ as additional elements of $\Gamma_K$ which are respectively smaller and bigger than any element in $\Gamma_K$. Furthermore we have an infinity element ``$\infty$'' (without sign) which we treat as an additional point of $K$ which satisfies the following: a set $X\subseteq K$ contains $\infty$ if and only there is $\gamma\in \Gamma_K$ such that $K\setminus B(0,\gamma)\subseteq X$. Thus, by a \emph{neighborhood of infinity} we mean a set of the form $K\setminus B(0,\gamma)$ for some $\gamma\in \Gamma_K$. We say that a family of functions $f\colon X\subseteq W\times K\to K$ is \emph{defined at a neighborhood of infinity} if for all $w\in W$, the function $f_w$ is defined at a neighborhood of infinity. 

\section{$\Gamma$-factorization}

In what follows we use the following notation: for $\delta\in\Gamma_K\cup\{-\infty\}$, let $\Gamma_K^{>\delta}$ denote the set $\Gamma_K^{>\delta}:=\{\gamma\in\Gamma_K: \gamma>\delta\}$ and analogously for $\Gamma_K^{<\delta}$ with $\delta\in\Gamma_K\cup\{+\infty\}$.

\begin{theorem}[$\Gamma$-factorization I]\label{thm:gamma-fact-infty} Let $(K,\cL)$ be $C$-minimal and $f\colon X\subseteq K\to K^\times$ be a definable function defined at a neighborhood of infinity. Then $f$ factorizes at infinity over $\Gamma$ through a definable function $h$.  
\end{theorem}

\begin{proof} 
Consider the function $g\colon \Gamma_K\to\Gamma_K\cup\{-\infty\}$ given by 
\begin{equation*}
h(\gamma):=\inf\{v(f(x)) : v(x)=\gamma\}.
\end{equation*}
By o-minimality (part (1) of Theorem \ref{lem:omin}), the function $h$ is well-defined.  

\begin{claim}\label{claim:notinfty} There is $\delta_1\in\Gamma_K$ such that $h(\gamma)\in \Gamma_K$ for all $\gamma<\delta_1$.  
\end{claim} 
For $\gamma,\mu\in \Gamma_K$, define
\[
D_{\gamma,\mu} := f^{-1}(v^{-1}(]-\infty,\mu])) \cap (v^{-1}(\gamma)), 
\]
and let $Y:=\{\gamma\in \Gamma_K: \forall \mu\in \Gamma_K \ (D_{\gamma,\mu}\neq\emptyset)\}$. Fix $\gamma\in Y$ (if $Y\neq \emptyset)$. For each $\mu\in \Gamma_K$, the set $D_{\gamma,\mu}$ is contained in $B(0,\gamma)\setminus B^\circ(0,\gamma)$. By part (1) of Theorem \ref{lem:omin} and the fact that definable families of subsets of $K$ are families of Swiss cheeses (see Remark \ref{rmk:swiss}), there is an integer $n$ such that, coinitially in $\Gamma_K$, only one of the following happens 
\begin{enumerate}
\item $D_{\gamma,\mu}$ contains a set of the form $B(0,\gamma)\setminus \bigcup_{i=1}^{n} B_{i,\mu}$, or
\item  $D_{\gamma,\mu}$ is contained in $\bigcup_{i=1}^{n} B_{i,\mu}$,  
\end{enumerate}
where the balls $B_{1,\mu},\ldots,B_{n,\mu}$ are of the form $B_{i,\mu}=B^\circ(a_{i,\mu},\gamma)$ with $v(a_{i,\mu})=\gamma$. 

Let us first show that for all $\gamma\in Y$, (1) cannot hold coinitially in $\Gamma_K$. For suppose for a contradiction that there is $\gamma\in Y$ such that (1) holds coinitially in $\Gamma_K$. Since $D_{\gamma,\mu'}\subseteq D_{\gamma,\mu}$ for all $\mu'\leqslant \mu$, by part (1) of Theorem \ref{lem:omin}, the balls $B_{i,\mu}$ can be taken equal coinitiallly in $\Gamma_K$. But this implies that the intersection $\bigcap_{\mu\in\Gamma_K} D_{\gamma,\mu}\neq\emptyset$, which is a contradiction since for every $x\in B(0,\gamma)$ there is $\mu$ such that $x\notin D_{\gamma,\mu}$. 

Therefore, for all $\gamma\in Y$, we must have that (2) holds coinitially in $\Gamma_K$. So for every $\gamma\in Y$, there is an element $\mu_{\gamma}\in \Gamma_K$ such that for all $\mu<\mu_{\gamma}$, the set $D_{\gamma,\mu}$ is contained in a union of $n$ open balls $B^\circ(a,\gamma)$ with $v(a)=\gamma$. Consider the definable set 
\[
\bigcup_{\gamma\in Y, \mu < \mu_{\gamma}} D_{\gamma,\mu}.
\] 
If $Y$ is infinite, we contradict $C$-minimality, as such set is not a finite disjoint union of Swiss cheeses. Therefore, $Y$ is finite. Let $\delta_1$ be the minimal element of $Y$ if $Y\neq\emptyset$ or $\delta_1=0$ if $Y=\emptyset$. By construction, $\delta_1$ satisfies the claim. 

\

For $\gamma\in \Gamma_K$, define 
\[
A_\gamma := \{ x \in K : v(x)=\gamma \wedge v(f(x))=h(\gamma) \}.
\] 
By Claim \ref{claim:notinfty}, the union $\bigcup_{\gamma\in\Gamma_K} A_\gamma$ contains elements of arbitrarily small valuation. Therefore, by $C$-minimality, it contains a set of the form $K\setminus B(0,\gamma_0)$ with $\gamma_0\in \Gamma_K$. Let $\delta_2$ be the maximal such $\gamma_0$, if existing, or $\delta_2=0$ otherwise. To complete the proof, we show that $v(f(x))=h(v(x))$ for all $x\in K\setminus B(0,\delta_2)$. Indeed, if $x\in K\setminus B(0,\delta_2)$, then $x\in A_{\gamma}$ for some $\gamma\in \Gamma_K$. But by the definition of $A_\gamma$ this only holds if $v(x)=\gamma$ and $v(f(x))=h(x)$.  
\end{proof}

\begin{theorem}[$\Gamma$-factorization II]\label{thm:local-gamma-fact} Let $(K,\cL)$ be a $C$-minimal valued field and let $f\colon X\subseteq K\to K$ be a definable local $C$-isomorphism. Then there is a finite subset $F\subseteq X$ such that $f| (X\setminus F)$ locally factorizes over $\Gamma$ through a definable function $h$.  
\end{theorem}

\begin{proof} Let $x\in X$ and $B_x=B^\circ(x,\gamma_0(x))$ be maximal such that $f_{|B_x}$ is a $C$-isomorphism. Note that $\gamma_0\colon X\to \Gamma_K\cup\{-\infty\}$ is definable. The definition of $C$-isomorphism implies that the function $h_{x}\colon \Gamma_K^{>\gamma_0(x)}\rightarrow \Gamma_K$ defined by 
\begin{equation*}
h_{x}(\gamma)= v(f(x)-f(y)) \text{ for some (all) $y$ such that $v(x-y)=\gamma$,}
\end{equation*} 
is well defined. For $\square\in\{=,<,>\}$, we write $[h_x] \square [h_y]$ to express the following property 
\[
[h_x] \square [h_y] \Leftrightarrow (\exists\delta\in\Gamma)(\forall \gamma>\delta) (h_x(\gamma) \square h_y(\gamma)).
\]
Thus, $[h_x]=[h_y]$ means that the functions $h_x$ and $h_y$ are eventually equal. 

\begin{claim}\label{claim:1.1} There is a definable function $\gamma_1\colon X\to \Gamma_K\cup\{-\infty\}$ such that for all $x\in X$, $\gamma_1(x)\geqslant\gamma_0(x)$ and for all $y\in B^\circ(x,\gamma_1(x))\setminus \{x\}$, either $[h_x]=[h_y]$, $[h_x]>[h_y]$ or $[h_x]<[h_y]$. 
\end{claim}

Consider the definable sets $A_{\square}(x):=\{y\in B_x: [h_x]\square[h_y]\}$ where $\square\in\{=,<,>\}$. These sets are disjoint and, by o-minimality, they cover $B_x$. Therefore, by $C$-minimality, there is $\gamma_1(x)\geqslant \gamma_0(x)$ minimal such that $B^\circ(x,\gamma_1(x))\setminus\{x\}$ is contained in one of them, which shows the claim.  

\

In what follows, set $B_x:=B^\circ(x,\gamma_1(x))$. 

\begin{claim}\label{claim:2} For all but finitely many $x\in X$, for all $y\in B_x$, $[h_x]=[h_y]$. 
\end{claim}

Suppose not. Then, by Claim \ref{claim:1.1}, there are infinitely many $x\in X$ for which for all $y\in B_x\setminus\{x\}$, either $[h_x]< [h_y]$ or $[h_x]>[h_y]$. Suppose that there are infinitely many $x\in X$ for which $[h_x]< [h_y]$ for all $y\in B_x\setminus\{x\}$ (the other case is analogous). By $C$-minimality, there is an open ball $B\subseteq X$ such that for all $x\in B$ for all $y\in B_x\setminus\{x\}$, $[h_x]<[h_y]$. But this implies a contradiction, since for $x,y\in B$ such that $x\neq y$ we have that $[h_x]<[h_y]<[h_x]$. This shows  the claim. 

\

Let $F\subseteq X$ be the finite set given by Claim \ref{claim:2} and let $x\in X\setminus F$. Let $g_x\colon B_x\to\Gamma_K\cup\{-\infty\}$ be the function sending $y$ to the smallest $\delta\in \Gamma_K\cup\{-\infty\}$ such that $(\forall \gamma>\delta) (h_x(\gamma) = h_y(\gamma))$, which exists since $[h_x]=[h_y]$. By part (2) of Theorem \ref{lem:omin}, the function $g_x$ is locally constant on $B_x\setminus W_x$, with $W_x$ a finite set. Let $\gamma_3\colon X\setminus F\to \Gamma_K\cup\{\infty\}$ be the definable function sending $x$ to the minimal $\gamma$ such that $B^\circ(x,\gamma)\subseteq B_x\setminus W_x$. Reset $B_x:=B^\circ(x,\gamma_3(x))$. Consider the set 
\[
A:=\{x\in X\setminus F: \neg(\exists\gamma\in \Gamma_K) (\gamma\geqslant\gamma_3(x) \wedge (\forall y,z\in B^\circ(x,\gamma)\setminus \{x\})(g_x(y)=g_x(z))\},
\] 
consisting of elements $x \in X\setminus F$ for which there is no open ball $B$ around $x$ such that $g_x$ restricted to $B\setminus \{x\}$ is constant. 
\begin{claim} The set $A$ is finite. 
\end{claim}

Suppose not. Since $A$ is definable, by $C$-minimality, there is an open ball $B\subseteq A$. Pick $x\in B$ and $y\in B\cap B_x$. Then, since $y\in B_x$, there is an open ball $B'\subseteq B\cap B_x$ containing $y$ and such that $g_x$ is constant on $B'$. We claim that $g_y$ is defined and constant on $B'\setminus\{y\}$, which contradicts that $y\in A$. It suffices to show that there is some $\delta$ such that for all $z\in B'\setminus\{y\}$, $h_y(\gamma)=h_z(\gamma)$ for all $\gamma>\delta$. Indeed, if such $\delta$ exists, then there is a minimal such element which would be equal to $g_y(z)$ for all $z\in B'\setminus\{y\}$. We take $\delta=g_x(z)=g_x(y)$, which is well defined since by assumption $g_x$ is constant on $B'$. By definition of $g_x$ we have that for all $\gamma>g_x(z)$
\[
h_z(\gamma)=h_x(\gamma)=h_y(\gamma), 
\]    
which completes the claim. 

\

Let $x\in X\setminus (A\cup F)$. Since $x\notin A$, let $\gamma_4\colon  X\setminus (A\cup F)\to\Gamma_K\cup\{-\infty\}$ be the definable function sending $x$ to the minimal $\gamma\geqslant\gamma_3(x)$ such that $g_x$ is constant on $B^\circ(x,\gamma)\setminus\{x\}$. Finally, reset $B_x$ to now denote the maximal open ball $B^\circ(x,\gamma_5(x))$, where 
\[
\gamma_5(x)=\max\{\gamma_4(x),g_x(z)\} \text{ for some (all) $z\in B^\circ(x,\gamma_4(x))\setminus\{x\}$}.
\]
Then, $f_{|B_x}$ factorizes over $\Gamma$ through $h_x$. 
\end{proof}

\subsection*{Uniform results} The reader can check that the proofs of both Theorem \ref{thm:gamma-fact-infty} and \ref{thm:local-gamma-fact} are uniform in parameters. We provide the exact statements of what we mean by `uniform in parameters' for the reader's convenience. 

\begin{theorem}\label{thm:uniform-gamma-fact-infty} Let $(K,\cL)$ be $C$-minimal and $f\colon X\subseteq W\times K\rightarrow K^\times$ be a definable family of functions defined at a neighborhood of infinity. Then there is a definable family of functions $h\colon Y\subseteq W\times \Gamma_K\to \Gamma_K$ such that for all $w\in W$, $f_w$ factorizes at infinity over $\Gamma$ through $h_{w}$. 
\end{theorem}
%

\begin{theorem}\label{thm:uniform-local-gamma-fact} Let $(K,\cL)$ be $C$-minimal and $f\colon X\subseteq W\times K\to K$ be a definable family of local $C$-isomorphisms. Then there are a definable family of functions $h\colon Y\subseteq X\times \Gamma_K\to \Gamma_K$ and a definable subset $F\subseteq X$ such that, for all $w\in W$ the set $F_w$ is finite, and $f_w|(X_w\setminus F_w)$ locally factorizes over $\Gamma$ through $h_{w}$. 
\end{theorem}

%
%
%

\section{Polynomially bounded $C$-minimal valued fields}

Let us start by defining the main concepts of this section:

\begin{definition}\label{def:linbdd} Let $(R,<,+)$ be a divisible ordered abelian group definable inside some structure $(M,\cL)$. The structure $(R,\cL_{ind})$ is said to be \emph{linearly bounded} if for every $\cL_{ind}$-definable function $f\colon R\to R$ there are a definable endomorphism $\lambda\in End(R,<,+)$ and $a\in R$ such that $|f(x)|<|\lambda(x)|$ for all $x$ such that $|x|>|a|$. It is said to be \emph{$\cL$-uniformly linearly bounded} if for every $\cL$-definable set $W$ and every $\cL$-definable family of functions $f\colon W\times R\to R$, there are an $\cL$-definable $\lambda\in End(R,<,+)$ and an $\cL$-definable function $a\colon W\to R$ such that, for all $w\in W$, $|f_w(x)|<|\lambda(x)|$ for all $x>a(w)$. It is $\Q$-linearly bounded (resp. $\cL$-uniformly $\Q$-linearly bounded) if in addition $\lambda$ can be chosen to be $\lambda(x)=nx$ for some integer $n$.   
\end{definition}

\begin{definition}\label{def:polybdd} An expansion $(K,\cL)$ of $(K,\cL_{div})$ is said to be \emph{uniformly polynomially bounded}, if for every definable set $W$ and definable family of functions $f\colon X\subseteq W\times K\rightarrow K$, there is an integer $n$ and a definable function $a\colon W\to \Gamma_K$ such that for each $w\in W$, $v(f_w(x)) > n v(x)$ for all $x\in X$ such that $v(x)<a(w)$.  
\end{definition}

We will need the following dichotomy due to Miller and Starchenko in o-minimal expansions of ordered groups. 

\begin{theorem}[{Miller-Starchenko \cite[Theorems A and B]{miller-starchenko98}}]\label{thm:miller} Suppose that $(R,\cL)$ is an o-minimal expansion of an ordered group $(R, <, +)$. Then exactly one of the following holds: 
\begin{enumerate}[(a)]
\item $(R,\cL)$ defines a binary operation $\cdot$ such that $(R, <, +, \cdot)$ is an ordered real closed field or 
\item for every definable $\alpha \colon  R \to R$ there exist $c\in R$ and a definable $\lambda \in \{0\} \cup Aut(R, +)$ with $\lim_{x\to +\infty} [\alpha(x) -\lambda(x)] = c$.
\end{enumerate} 
If (b) holds and there is a distinguished $\emptyset$-definable element $1>0$ then every definable endomorphism of $(R,+)$ is $\emptyset$-definable. 
\end{theorem}

\begin{remark}\label{rmk:miller-lin-bdd} Note that condition (b) implies that the structure $(R,\cL)$ is linearly bounded. Indeed one may always suppose without loss of generality that there is a distinguished $\emptyset$-definable element $1>0$.
\end{remark}

\begin{corollary}\label{cor:Qlinearly} Every o-minimal expansion of $(\Q,<,+,0,1)$ is $\Q$-linearly bounded. 
\end{corollary} 

\begin{proof} Since multiplication is not definable in any o-minimal expansion of $(\Q,<,+,0,1)$, by Theorem \ref{thm:miller} (and the previous remark), for any definable $\alpha\colon\Q\to\Q$, there exist a $\emptyset$-definable $\lambda \in \{0\} \cup Aut(\Q, +)$ and $a\in \Q$ satisfying that $|\alpha(x)|<|\lambda(x)|$ for all $x$ such that $|x|>|a|$. The result follows from the fact that every automorphism of $(\Q, +)$ is of the form $qx$ for $q\in \Q$.  
\end{proof}

\begin{theorem}\label{linear-to-uniform} Let $R$ be a definable (or interpretable) and stably embedded set in a structure $(M,\cL)$. Suppose moreover that $(R,\cL_{ind})$ is an o-minimal expansion of an ordered divisible group $(R,<,+,0,1)$. Then $(R,\cL_{ind})$ is linearly bounded if and only if is $\cL$-uniformly linearly bounded. Moreover, it is $\Q$-linearly bounded if and only if it is $\cL$-uniformly $\Q$-linearly bounded.  
\end{theorem}

\begin{proof} One direction is trivial, so suppose $(R,\cL_{ind})$ is linearly bounded and let $f\colon W\times R\to R$ be an $\cL$-definable family of functions. By Theorem \ref{thm:miller}, the two following conditions are equivalent 
\begin{enumerate}
\item $(R,\cL_{ind})$ is linearly bounded;
\item no field expanding $(R,+)$ is $\cL_{ind}$-definable.
\end{enumerate}
Therefore, the assumption implies that no field expanding $(R,+)$ is $\cL_{ind}$-definable. Since this last condition is expressible by an $\cL_{ind}$-axiom scheme, Theorem \ref{thm:miller} implies that every $\cL_{ind}$-elementary equivalent structure is linearly bounded. Furthermore, by Theorem \ref{thm:miller}, every $\cL_{ind}$-definable endomorphism of $(R,+)$ is $\emptyset$-definable. Let $\Lambda$ be the ring of all $\cL_{ind}$-definable endomorphisms of $(R,<,+)$. By stable embeddedness, for every $w\in W$, the function $f_w\colon R\to R$ is $\cL_{ind}$-definable. Our assumption implies there are a $\emptyset$-definable endomorphism $\lambda_w\in\Lambda$ and $a_w\in R$ such that $|f_w(x)|<|\lambda(x)|$ for all $x$ such that $|x|>|a_w|$ (see Remark \ref{rmk:miller-lin-bdd}). Since $\Lambda$ has bounded cardinality, by compactness and the fact that every $\cL_{ind}$-elementary equivalent structure is linearly bounded,  there is $\lambda\in\Lambda$ such that 
\[
(\forall w\in W)(\exists a_w\in R)(\forall x)(|x|> |a_w|\to |f_w(x)| < |\lambda(x)|, 
\]
which shows that $(R,\cL_{ind})$ is $\cL$-uniformly bounded. 

If furthermore $\Gamma_K$ is $\mathbb Q$-linearly bounded, then every $\lambda\in \Lambda$ is bounded by a function of the form $\lambda(x)=nx$ for $n\in \N$. Therefore $(R,\cL_{ind})$ is $\cL$-uniformly $\Q$-linearly bounded. 
\end{proof} 

\begin{theorem}\label{thm:germs} Suppose $(K,\cL)$ is a $C$-minimal valued field such that $(\Gamma_K,\cL_{ind})$ is $\mathbb Q$-linearly bounded. Then $(K,\cL)$ is uniformly polynomially bounded.  
\end{theorem}
\begin{proof} 
Let $f\colon X\subseteq W\times K\rightarrow K^\times$ be a definable family of functions defined at a neighborhood of infinity. By Theorem \ref{thm:uniform-gamma-fact-infty}, let $h\colon Y\subseteq W\times \Gamma_K\to \Gamma_K$ be a definable $\Gamma$-factorization of $f$ at infinity, that is, for every $w\in W$, $f_w$ factorizes at infinity over $\Gamma$ through $h_w$. It suffices to show that $h$ is $\cL$-uniformly $\Q$-linearly bounded at infinity. Since $C$-minimality is preserved by addition of constants, we may assume there is a $\emptyset$-definable element $1>0$ in $\Gamma_K$. The result now follows by Theorem \ref{linear-to-uniform} taking $(M,\cL)=(K,\cL)$ and $R=\Gamma_K$.

\end{proof}

Using Corollary \ref{cor:Qlinearly} and Theorem \ref{thm:germs}, we obtain the following: 

\begin{theorem}\label{main} Let $(K,\cL)$ be a $C$-minimal valued field with $\Gamma_K=\Q$. Then $(K,\cL)$ is uniformly polynomially bounded. In particular, any $C$-minimal expansion of $\C_p$ or $\overline{\mathbb{F}_p}^{\mathrm{alg}}(\!(t^\Q)\!)$ is polynomially bounded. More generally, any $C$-minimal valued field $(K,\cL)$ which is $\cL$-elementary equivalent to a valued field having as value group $\Q$, is uniformly polynomially bounded. 
\end{theorem}

We end up this section by giving an example of a polynomially bounded $C$-minimal valued field in which definable $\Gamma$-functions are not $\Q$-linearly bounded.  

\begin{example} Consider the algebraically closed valued field $K=\C(\!(t^{\R})\!)$ having value group $\R$. Consider the two-sorted structure $(\mathcal{K},\cL)$
\[
\mathcal{K}:=
\begin{cases}
(K,\Lring)  \\
(\R, \cL_{or})\\
v\colon K\to \R\cup\{\infty\}\\
\end{cases}
\]
where $\cL_{or}:=\Lring\cup\{<\}$. The theory of $\mathcal{K}$ has elimination of quantifiers (one follows the same argument as for the reduct in which the value group has only the ordered abelian group structure, see \cite{haskell-hrushovski-macpherson2008}). Since $(\R, \cL_{or})$ is o-minimal, it is easy to see that definable sets in one valued field variable can only be finite boolean combinations of balls (in any model of $Th(\mathcal{K})$). Consider the one sorted language $\cL^*$ containing $\cL_{div}$ and a relation symbol for every $\cL$-definable subset of $K^n$ (for all $n$). Interpreting the language in the obvious way, we have that $(K,\cL^*)$ is $C$-minimal. Clearly, there are definable unary $\Gamma$-functions which are not linearly bounded since we have multiplication in the value group. Nevertheless, it is worthy to note that $(K,\cL^*)$ is polynomially bounded. Indeed, by quantifier elimination, every $\cL^*$-definable function $f\colon K\to K$ is already $\cL_{div}$-definable. Therefore, although the fact of having all definable $\Gamma$-functions to be $\Q$-linearly bounded is a sufficient condition for a $C$-minimal valued field to be polynomially bounded, the example shows it is not a necessary condition. We finish this section with the following natural question: 
\end{example}

\begin{question} Is every $C$-minimal valued field polynomially bounded? 
\end{question}

\section{$\RV$-factorization}

Let us start by defining what eventually linear means. By a unary $\Gamma$-function in $K$ we simply mean a function $g\colon \Gamma_K\to \Gamma_K$. 

\begin{definition}\label{def:eventually} Let $g\colon \Gamma_K\to \Gamma_K$ be a unary $\Gamma$-function. 
\begin{enumerate}
\item We say $g$ is \emph{eventually linear} if there are $\delta,\alpha\in \Gamma_K$ and an endomorphism $\lambda$ of $(\Gamma_K,+)$ such that $g(\gamma)= \lambda(\gamma)+\alpha$ for all $\gamma\in \Gamma_K^{>\delta}$. 
\item We say $g$ is \emph{eventually $\Q$-linear} if in the previous definition, $\lambda$ can be taken of the form $\lambda(\gamma)= r\gamma$ for $r\in \Q$. 
\item Given a definable set of parameters $W$, a definable family of functions $f\colon W\times \Gamma_K\to\Gamma_K$ is \emph{eventually $\Q$-linear} if there is a finite subset $Z_f\subseteq\mathbb{Q}$ such that for every $w\in W$, the function $f_w$ is eventually $\Q$-linear with slope $r\in Z_f$.
\end{enumerate}
  \end{definition}

In algebraically closed valued fields, every $\cL_{div}$-definable $\Gamma$-function is already definable in the language of ordered groups $\cL_{og}:=\{<,+,0\}$ (see for instance \cite[Theorem 2.1.1]{haskell-hrushovski-macpherson2008}). It follows, using quantifier elimination of divisible ordered abelian groups, that unary $\cL_{div}$-definable $\Gamma$-functions are eventually $\Q$-linear. The same holds for $C$-minimal expansions with analytic structure as studied in \cite{lipshitzETAL:98}, as the induced structure on $\Gamma_K$ is again the pure structure of an ordered abelian group.  

\begin{remark}\label{rmk:even-comp} Let $g\colon \Gamma_K\to \Gamma_K$ be a definable function. It follows from the assumption that unary definable $\Gamma$-functions are eventually linear, that there are $\delta',\alpha'\in \Gamma_K$ and $\lambda$ in $End(\Gamma_K,+)$ such that for all $\gamma\in \Gamma_K^{<\delta'}$ 
\[g(\gamma)= \lambda(\gamma)+\alpha'.\]
\end{remark}

Given a function $f\colon X\subseteq K\to K$, we adopt the convention concerning limits in $K\cup\{\infty\}$: 
\begin{enumerate}
\item[$\bullet$] for $a\in K$, $\displaystyle{\lim_{x\to a} f(x)=\infty}$ holds if for every $\gamma\in\Gamma_K$, there is an open ball $B$ containing $a$ such that $v(f(B))<\gamma$;
\item[$\bullet$] for $b\in K$, $\displaystyle{\lim_{x\to \infty} f(x)=b}$ holds if for every $\gamma\in\Gamma_K$, there is $\delta\in \Gamma_K$ such that $v(f(x)-b)>\gamma$ for all $x$ for which $v(x)<\delta$; 
\item[$\bullet$] $\displaystyle{\lim_{x\to \infty} f(x)=\infty}$ holds if for every $\gamma\in\Gamma_K$, there is $\delta\in \Gamma_K$ such that $v(f(x))<\gamma$ for all $x$ for which $v(x)<\delta$;  
\item[$\bullet$] for $a, b\in K$, $\displaystyle{\lim_{x\to a} f(x)=b}$ is the usual definition. 
\end{enumerate}

The following Lemma follows from the proof of \cite[Proposition 4.5]{delonCorps:12} with minor modifications. It ensures the existence of limits in definably complete $C$-minimal valued fields for which all definable unary $\Gamma$-functions are eventually linear. 

\begin{lemma}\label{lem:delon1} Let $(K,\cL)$ be a definably complete $C$-minimal valued field in which definable unary $\Gamma$-functions are eventually linear. Let $f\colon X\subseteq K\to K$ be a definable function. Then, for any $a\in \overline{X}$ (the topological closure of $X$ in $K\cup\{\infty\}$), the limit $\displaystyle{\lim_{x\to a} f(x)}$ exists in $K\cup\{\infty\}$. 
\end{lemma}

We have all the elements to state and prove our $\RV$-factorization results. Recall that $p$ denotes the characteristic exponent of $K$ (that is, $p=\mathrm{char}(K)$ if $\mathrm{char}(K)>0$ and $p=1$ otherwise). We will directly show them in families. 

\begin{theorem}[$\RV$-factorization I]\label{thm:LJPatinfty} Let $(K,\cL)$ be a $C$-minimal valued field and $f\colon X\subseteq W\times K\rightarrow K^\times$ be a definable family of functions defined at a neighborhood of infinity. Let $h\colon Y\subseteq W\times \Gamma_K\to \Gamma_K$ be a $\Gamma$ factorization of $f$ at infinity, namely, for every $w\in W$, $f_w$ factorizes at infinity over $\Gamma$ through $h_w$. If $h$ is eventually $\Q$-linear, then there are a finite definable partition of $W$ into sets $W_1\cup\dots\cup W_\ell$, integers $n_1,\ldots,n_\ell,m_1,\ldots,m_\ell \in \mathbb Z$ and a definable function $c\colon W\to \RV^*$, such that for all $w\in W_i$, in a neighborhood of infinity, 
\[
\rv(f_w(x)) = \rv(x)^{n_i/p^{m_i}} c(w).
\] 
If in addition $(K,\cL)$ is definably complete and all definable unary $\Gamma$-functions are eventually linear, for all $w\in W_i$, the limit 
\[
a(w):=\lim_{x\to \infty} \frac{f_w(x)}{x^{n_i/p^{m_i}}}
\]
exists in $K$ and $c(w)=\rv(a(w))$. 
\end{theorem}

\begin{proof} By assumption, there are rational numbers $r_1,\ldots, r_\ell$ and definable functions $\alpha\colon W\to \Gamma_K$ and $\delta\colon W\to \Gamma_K\cup\{+\infty\}$ such that for each $w\in W$ there is $i\in \{1,\ldots,\ell\}$ satisfying that for all $\gamma<\delta(w)$
\[
v(f_w(x))=h_{w}(\gamma)= r_i \gamma + \alpha(w) \text{ for all $x$ such that $v(x)=\gamma$}.
\] 
This induces naturally a definable partition of $W = W_1\cup\dots\cup W_\ell$ depending on the slopes $r_i$. From now on we work over $W=W_i$ and let $r$ be $r_i$. Without loss of generality, suppose that $r\geqslant 0$. Let $r=s{t^{-1}}$, with $s$ and $t$ coprime positive integers or $s=0$ and $t=1$. Fix $w\in W$ and let $b\in K$ be such that $v(b)=\alpha(w)$. Thus, we have that 
\[
v(f_w(x)^{t})=v(x^{s}b^{t}) \text{ for all $x$ such that $v(x)<\delta(w)$}. 
\]
Consider the definable family of sets $A^{w,b}\subseteq \Gamma_K^{<\delta(w)}\times (K/v)^\times$ defined by having fibers of the form
\[
A_\gamma^{w,b}:=\{[f_w(x)^{t}x^{-s}b^{-t}]/v : \gamma<\delta(w), v(x)=\gamma\}.
\]
Let $T^{w,b}\colon  K/v\to \Gamma_K\cup\{-\infty\}$ be the definable partial function given by 
\[
c/v\mapsto \gamma_{c}:=\inf \{\gamma\in \Gamma_K^{<\delta(w)}\cup\{-\infty\}: (c/v)\in A_\gamma^{w,b}\}.
\]
By o-minimality, $T^{w,b}$ is well-defined and by Theorem \ref{lem:omin} (part (3)), $T^{w,b}$ has finite image. Moreover, the image of $T^{w,b}$ is coinitial in $\Gamma_K$, so it must attain the value $-\infty$. Therefore, again by o-minimality, there are $c\in K$ and $\delta'(w)\leqslant \delta(w)$ such that $(c/v)\in A_\gamma^{w,b}$ for all $\gamma<\delta'(w)$. By quantifying over $b$ and $c$, we can suppose that $\delta'\colon W\to \Gamma_K\cup\{+\infty\}$ is a definable function: we define $\delta'(w)$ to be be the minimal element $\gamma_1\in \Gamma_K\cup\{+\infty\}$ such that 
\begin{equation}\label{eq:defgamma}\tag{E1}
\gamma_1\geqslant \delta(w) \wedge (\exists b\in K) (\exists c\in \mathcal{O}_K) (v(b)=\alpha(w) \wedge (\forall \gamma<\gamma_1) ((c/v)\in A_\gamma^{w,b})). 
\end{equation}
For any $b,c\in K^\times$ satisfying \eqref{eq:defgamma}, we have that 
\begin{equation}\label{eq:rvorder}\tag{E2}
\rv(f_w(x))^t = \rv(x)^{s}\rv(b^tc),
\end{equation} 
for all $x\in K\setminus B(0,\delta'(w))$. Let us show that $t=p^{m}$ for some integer $m\in \NN$ (so we assume that $s\neq 0$). Since $s$ and $t$ are positive, by $C$-minimality we have that $f_w(K\setminus B(0,\delta'(w))$ contains a neighborhood of infinity. Let $t=t'p^{v_p(t)}$ and $s=s'p^{v_p(s)}$, where $v_p$ denotes the $p$-adic valuation when $p>1$, and $v_1$ denotes the constant 0 function. Now, for $x$ in a neighborhood of infinity consider the following sets:
\begin{itemize}
\item $A_x:=\{\rv(y): \rv(f_w(x))=\rv(f_w(y))\wedge v(x)=v(y)\} $,
\item $B_x:=\{\rv(y): \rv(f_w(x))^t=\rv(f_w(y))^t\wedge v(x)=v(y)\}$ and
\item $C_x:=\{\rv(y): \rv(x)^s=\rv(y)^s\}$. 
\end{itemize}
By $C$-minimality, there is $d\in \mathbb{N}^\ast\cup\{+\infty\}$ such that, for all $x$ in a neighborhood of infinity, $|A_x|=d$. Similarly, for all $x$ in a neighborhood of infinity, we have that $|B_x|=dt'$ and $|C_x|=s'$. By \eqref{eq:rvorder}, we have that $B_x=C_x$, and therefore $d\in \NN^*$ and $dt'=s'$. This shows that $t'$ divides $s'$. Since $s$ and $t$ are coprime, we must have that $t'=1$, which shows what we wanted.   

Define $c(w):=\rv(bc^{1/p^m})$, for any $b,c\in K$ satisfying \eqref{eq:defgamma}. We have thus that $\rv(f_w(x)) = \rv(x)^{s/p^m} c(w)$. 

\

To show the last statement, by Lemma \ref{lem:delon1} there is $a(w)\in K\cup\{\infty\}$ such that  
\[
a(w):=\lim_{x\to \infty} \frac{f_w(x)}{x^r}. 
\]
Since in a neighborhood of infinity all values of $f_w(x)x^{-r}$ lie in a single ball of radius $0$, we must have that $a(w)\in K$. Clearly, $\rv(a(w))=c(w)$. 
\end{proof}

In what follows $p$ will still denote the characteristic exponent of $K$ and $q\geqslant 0$ will denote the characteristic of the residue field $K/v$. 

\begin{theorem}[$\RV$-factorization II]\label{thm:almostjac} Let $(K,\cL)$ be a $C$-minimal valued field and~$f\colon X\subseteq W\times K\rightarrow K$ be a definable family of local $C$-isomorphisms. Suppose $h\colon Y\subset X\times \Gamma_K \rightarrow \Gamma_K$ is a local $\Gamma$-factorization of $f$ (namely, for all $w\in W$, $f_w$ locally factorizes over $\Gamma$ through $h_w$) which is eventually $\Q$-linear. Then, there is a finite definable partition of $X$ into sets $X_1\cup\dots\cup X_\ell$,  integers $n_1,\ldots,n_\ell$ and definable functions $\delta\colon X\to\Gamma_K$ and $c\colon X \to \RV^*$ such that for all $(w,x)\in X_i$, $B^\circ(x,\delta(w,x))\subseteq (X_i)_w$ and for all distinct $y,z\in B^\circ(x,\delta(w,x))$, 
\[
\rv((f_w(y)-f_w(z))=\rv((y-z))^{p^{n_i}} c(w,x).
\] 
In particular, when $p=1$ we may assume that $\ell=1$. If in addition $(K,\cL)$ is definably complete and all definable unary $\Gamma$-functions are eventually linear, the limit 
\[
a(w,x):=\lim_{y\to x} \frac{f_w(x)-f_w(y)}{(x-y)^{p^{n_i}}}
\]
exists in $K$ and $c(w,x)=\rv(a(w,x))$. 
\end{theorem}

\begin{proof} 
Let $\theta\colon X\to \Gamma_K\cup\{-\infty\}$ be the definable function sending $(w,x)\in X$ to the minimal $\gamma$ such that $f_w|B^\circ(x,\gamma)$ is a $C$-isomorphism and $f_w| B^\circ(x,\gamma)$ factorizes over $\Gamma$ through $h_{w,x}$. Therefore, for every $(w,x)\in X$, all distinct $y,z\in B^\circ(x,\theta(w,x))$ and all $\gamma>\theta(w,x)$ we have that 
\begin{equation*}
h_{w,x}(\gamma)= v(f_w(z)-f_w(y)) \text{ if and only if $v(y-z)=\gamma$.}
\end{equation*} 
Since $h$ is eventually $\Q$-linear, there is a finite set $Z_h\subseteq \Q$ satisfying that for every $(w,x)\in X$, there is some $r\in Z_h$ such that for all $\gamma>\theta(w,x)$ (possibly replacing $\theta$ by a definable function taking higher values) 
\begin{equation}\label{eq:function}\tag{E3}
h_{w,x}(\gamma)=r\gamma+\alpha(w,x)
\end{equation}
where $\alpha\colon X\to \Gamma_K$ is a definable function. By Theorem \ref{thm:uniform-local-gamma-fact} we have moreover that for all $y\in B^\circ(x,\theta(w,x))$ and all $\gamma>\theta(w,x)$
\begin{equation}
h_{w,y}(\gamma)=r\gamma+\alpha(w,x).
\label{eq1:1}\tag{E4}
\end{equation}
Suppose that $Z_h=\{r_1,\ldots,r_\ell\}$. We partition $X$ into sets $X_1,\ldots, X_\ell$ where $X_i$ is given by 
\[
\{(w,x)\in X: (\forall \gamma>\theta(w,x)) (h_{w,x}(\gamma)=r_i\gamma+\alpha(w,x))\}. 
\]
From now on we suppose that $X$ is some $X_i$ and drop the indices (so $r$ will denote the slope $r_i$). 

Fix $(w,x)\in X$. Since $f_{w}|B^\circ(x, \theta(w,x))$ is a $C$-isomorphism we must have that $r>0$. Write $r=st^{-1}$ with $s$ and $t$ coprime positive integers and let $b\in K$ be such that $v(b)=\alpha(w,x)$. Unraveling definitions, we thus obtain that for all distinct $u,z\in B^\circ(x, \theta(w,x))$ 
\[
v((f_{w}(u)-f_{w}(z))^t)=v((u-z)^s \cdot b^t).
\]
Consider the definable family of sets $A^{w,x,b}\subseteq (\Gamma_K^{>\theta(w,x)})\times (K/v)^\times$ defined by having fibers
\[
A_\gamma^{w,x,b}:=\{[(f_{w}(u)-f_{w}(z))^t(u-z)^{-s}b^{-t}]/v : u, z\in B^\circ(x, \theta(w,x)), v(u-z)=\gamma\},
\]
and the definable partial function $T^{w,x,b}\colon  K/v\to \Gamma_K\cup\{+\infty\}$
\[
d\mapsto \gamma_d:=\sup \{\gamma\in \Gamma_K^{>\theta(w,x)}\cup\{+\infty\}: d\in A_\gamma^{w,x,b}\}.
\]
By o-minimality of $\Gamma_K$, $T^{w,x,b}$ is well-defined. By Theorem \ref{lem:omin} (part (3)),the image of $T^{w,x,b}$ has finite image. Therefore, since its image is cofinal in $\Gamma_K^{>\theta(w,x)}\cup\{+\infty\}$, $T^{w,x,b}$ must attain the value $+\infty$. Again by o-minimality, this implies that there are $c\in K$ and $\gamma_1\geqslant \theta(w,x)$ such that $(c/v)\in A_\gamma^{w,x,b}$ for all $\gamma>\gamma_1$. By quantifying over $b$ and $c$, we can make $\gamma_1$ definable over $(w,x)$: we let $\delta(w,x)$ be the minimal element $\gamma_1\in \Gamma_K\cup\{-\infty\}$ such that 
\begin{equation}\label{eq:defbc}\tag{E5}
\gamma_1\geqslant \theta(w,x) \wedge (\exists b\in K) (\exists c\in K) (v(b)=\alpha(w,x) \wedge (\forall \gamma>\gamma_1) ((c/v)\in A_\gamma^{w,x,b})). 
\end{equation}
Fix $b,c\in K$ satisfying \eqref{eq:defbc}.  We have that for all distinct $u,z\in B^\circ(x,\delta(w,x))$ 
\begin{equation}
(f_w(u)-f_w(z))^t\sim (u-z)^sb^tc.
\label{eq:sim2}\tag{E6}
\end{equation}

Recall that $q\geq 0$ denotes the characteristic of the residue field. 

\begin{claim}\label{claim:1}
Either $s=t=1$, or for some $k>0$ either $s=q^k$ and $t=1$, or $s=1$ and $t=q^k$.
\end{claim}

Pick $d\in K$ such that $v(d)>\delta(w,x)$, $e\in K$ a $t$-root of $d^{-s}b^{-t}c$ and consider the function defined on $\mathcal{O}_K$ given by 
\[g(z):=e(f_w(dz+x)-f_w(x)).\] 
Note that $g(0)=0$, $g$ is a $C$-isomorphism and moreover, for all $u,z\in \mathcal{O}_K$ we have that $(g(u)-g(z))^t\sim (u-z)^s$. Indeed, for $u,z\in \mathcal{O}_K$ we have that: 
\begin{align*}
(g(u)-g(z))^t 	& =  e^t(f_w(du+x)-f_w(dz+x))^t & \\
						&\sim e^t[(du+x)-(dz+x)]^s b^tc^{-1} & \text{ by  (\ref{eq:sim2}) }\\
						& = e^t d^s b^tc^{-1} (u-z)^s  = (u-z)^s. &
\end{align*}
This shows in particular that $v(g(u))=st^{-1}v(u)$ and that $g(u+\mathcal{M})\subseteq g(u)+\mathcal{M}$ for any $u\in \mathcal{O}_K$. Hence $g$ induces a residue map $\bar{g}\colon  K/v\to K/v$, which satisfies for all $u,z\in \mathcal{O}_K$
\[
(\bar{g}(u/v)-\bar{g}(z/v))^t=(u/v-z/v)^s.
\] 
We now work in $K/v$ (we will use $u,z\in K/v$ instead of $u/v$, $z/v$ for $u,z\in \mathcal{O}_K$). 

\

Since $(\bar{g}(u)-\bar{g}(z))^t=(u-z)^s$, we have both that $\bar{g}(u)^t=u^s$ and that $\bar{g}$ is injective. This implies, by Theorem \ref{lem:omin}, that there is a finite subset $F\subseteq K/v$ such that $\bar{g}(K/v)=(K/v)\setminus F$. If $q\neq 0$, we let $v_q$ denote the $q$-adic valuation on $\mathbb{Z}$. We split the argument in three cases: 

\

\textbf{Case 1} ($q\nmid s, q\nmid t$): The equality $\bar{g}(u)^t=u^s$ implies that for all $z\in K/v$, $u$ is an $s$-root of $z$ if and only if $\bar{g}(u)$ is a $t$-root of $z$. It is enough then to find some $z\neq 0$ such that all $t$-roots of $z$ are in $\bar{g}(K/v)$. For if this is true, $\bar{g}$ will be a bijection between all $s$-roots of $z$ and all $t$-roots of $z$ which implies $s=t$. From the assumption $gcd(s,t)=1$, we can conclude that $s=t=1$. To show such $z$ exists, consider the set $E=\{z\in K/v:\exists a\in F, a^t=z\}$. Given that $F$ is finite, so is $E$. Any $z\in (K/v)\setminus E$ satisfies the required property.  

\

\textbf{Case 2} ($q\neq 0, q\mid s, q\nmid t$): Let $k:=v_q(s)$ and $m\in \mathbb{N}$ such that $s=q^km$. This implies that $gcd(m,t)=1$ and our equality becomes $\bar{g}(u)^t=(u^{q^k})^m$. Thus for all $z\in K/v$, $u$ is an $m$-root of $z^{q^k}$ if and only if $\bar{g}(u)$ is a $t$-root of $z$. As in the previous case it is enough to find $z\neq 0$ such that all $t$-roots of $z$ are in $\bar{g}(K/v)$ to get a bijection from $m$-roots to $t$-roots, which implies $m=t=1$. The existence argument is exactly the same as in the previous case. 

\

\textbf{Case 3} ($q\neq 0, q\nmid s, q\mid t$): As in Case 2, let $k:=v_q(t)$ and set $m\in\mathbb{N}$ such that $t=q^km$. Thus, $gcd(m,s)=1$ and our equality becomes $(\bar{g}(u)^{q^k})^m=u^s$. Therefore for all $z\in K/v$, $u$ is an $s$-root of $z$ if and only if $\bar{g}(u)^{q^k}$ is an $m$-root of $z$. Let $\sigma(u)=u^{q^k}$. Here we need $z\neq 0$ such that all $s$-roots of $z$ are in $(\sigma\circ \bar{g})(K/v)$. Given that $\sigma$ is a bijection, a similar argument as in Case 1 shows the existence of such $z\in k$. Thus, $m=s=1$. This completes the claim. 

\

In all cases we will define the function $c\colon X\to \RV^*$ by $c(w,x):=rv(bc^{1/t})$ for any $b,c$ satisfying \eqref{eq:defbc}. Note that Claim \ref{claim:1} already implies the theorem for equicharacteristic valued fields. Moreover, if $p=1$ (recall $p$ is the characteristic exponent), then the only possible case is to have $s=t=1$, and therefore $Z_h=\{1\}$ (so assertion (2) of the theorem holds). For valued fields of mixed characteristic, it remains to show that the only possible case in the disjunction of Claim \ref{claim:1} is $s=t=1$. Since in Case 1 we already obtained that $s=t=1$, we may assume we are in either Case 2 or Case 3. We work again in $K$ and we assume that $1=p$ and $0<q$:

\

\textbf{(i)} Suppose we are in Case 2 above. Then there is some integer $k>0$ such that $s=q^k$, $t=1$ and therefore $g(u)\sim u^{q^k}$ for all $u\in \mathcal{O}_K$. Let $a\in \mathcal{O}_K\setminus\{0\}$ and $z\in K\setminus\{1\}$ such that $z^{q^k}=1$. This implies that $v(z)=0$ and thus that $az\in \mathcal{O}_K$ and $a\neq az$. Let $a_0\in \mathcal{O}_K\setminus\{0\}$ be such that 
\begin{enumerate}
\item $a_0^{q^k}\neq a^{q^k}$
\item $v(a-a_0)=v(az-a_0)=v(az-a)$.
\end{enumerate}   
Such an element $a_0$ always exists. Indeed, since $K/v$ is infinite, there are infinitely many elements $a_0$ satisfying (2), but at most $q^k$ roots of $a^{q^k}$. Now, on the one hand, the second condition on $a_0$ implies $\neg C(a_0,a,az)$. Since $g$ is a $C$-isomorphism, this implies $\neg C(g(a_0),g(a),g(az))$. On the other hand $C(a_0^{q^k},a^{q^k},a^{q^k})$ always holds, and since~$a^{q^k}=(az)^{q^k}$, we also have that~$C(a_0^{q^k}, a^{q^k}, (az)^{q^k})$. But $g(u)\sim u^{q^k}$ for all $u\in \mathcal{O}_K$, which contradicts Lemma \ref{lem:sim}. 

\

\textbf{(ii)} Suppose we are in Case 3 above. Then, for some integer $k>0$, $s=1$, $t=q^k$ and therefore $g(u)^{q^k}\sim u$ for all $u\in \mathcal{O}_K$. Since $g$ is a $C$-isomorphism and given that $g(0)=0$, $g(\mathcal{O}_K)$ is a closed ball centered at $0$. Given $u\in \mathcal{O}_K$ and $z=g(u)$, we have that 
\[
g(u)^{q^k}\sim u \Leftrightarrow z^{q^k}\sim g^{-1}(z)
\]
and get again the contradiction from case (i) for the $C$-isomorphism $g^{-1}$. This completes the proof of the first part. 

For the last assertion, suppose that $(K,\cL)$ is definably complete and all definable unary $\Gamma$-functions are eventually linear. Define $a(w,x)$ as 
\[
a(w,x):=\lim_{y\to x} \frac{f_w(x)-f_w(y)}{(x-y)^r}, 
\]
which exists by Lemma \ref{lem:delon1} (note that it cannot be the value $\infty$). Clearly, we have that $c(w,x)=\rv(a(w,x))$. 
\end{proof}

\section{The Jacobian Property}

Let $(K,\cL)$ be a definably complete $C$-minimal expansion of an algebraically closed valued field of characteristic 0 in which all definable unary $\Gamma$-functions are eventually $\Q$-linear. Note that by Lemma \ref{lem:delon1}, we have in this setting a well-defined notion of derivative for definable functions since limits exist. Let us recall the definition of the Jacobian property from \cite{cluckers-lipshitz:2011}.

\begin{definition}(Jacobian Property) \label{def:jacobian} Let $(K,v)$ be a valued field of characteristic $0$. Let $B\subseteq K$ be an open ball and $f \colon  B\to K$ be a function. We say that $f$ has the \emph{Jacobian Property} if 
\begin{enumerate}
\item[(i)] $f$ is injective and $f(B)$ is an open ball;
\item[(ii)] $f$ is differentiable and $f'$ is nonvanishing;
\item[(iii)] $rv(f'(x))$ is constant on $B$, say equal to $c$; 
\item[(iv)] for all $x, y\in B$, $rv(f(x)-f(y))=rv(x-y)c$.
\end{enumerate}
\end{definition}

\begin{theorem}\label{thm:LJP}
Let $(K,\cL)$ be a definably complete $C$-minimal valued field of characteristic $0$ in which definable unary $\Gamma$-functions are eventually $\Q$-linear. Let $f\colon X\subseteq W\times K\to K$ be a definable family of local $C$-isomorphisms. Then there is a set $F\subseteq W$ such that for all $w\in W$, $F_w$ is finite and $f_w|(X_w\setminus F_w)$ has locally the Jacobian property. 
\end{theorem}

\begin{proof}
By Theorem \ref{thm:almostjac}, there is a set $F\subseteq X$ such that, for each $w\in W$
\begin{enumerate}
\item $F_w$ is finite and
\item there are definable functions $\delta\colon W\to\Gamma_K$ and $c\colon X\to K$ such that for all $w\in W$, all $x\in X_w\setminus F_w$ and all distinct $y,z\in B^\circ(x,\delta(w,x))$, 
\[
rv(f_w(y)-f_w(z))= rv(y-z)c(w,x).
\] 
\end{enumerate} 
We left to the reader to check that $f_w|B^\circ(x,\delta(w,x))$ satisfies the Jacobian property: since $f_w|X_w$ is a $C$-isomorphism, condition (i) of Definition \ref{def:jacobian} is already satisfied; conditions (ii)-(iv) easily follow from point (2) above. 
\end{proof}

\begin{corollary}\label{cor:Jac}
Let $(K,\cL)$ be a definably complete $C$-minimal valued field of characteristic $0$ in which definable unary $\Gamma$-functions are eventually $\Q$-linear. Let $f\colon X\subseteq W\times K\to K$ be a definable family of functions. Then $X$ decomposes into definable sets $X=F\cup E\cup J$ such that for each $w\in W$: 
\begin{enumerate}
\item $F_w$ is a finite set; 
\item $f_w|_{E_w}$ is a locally constant function;
\item $f_w|_{J_w}$ has locally the Jacobian property. 
\end{enumerate}
\end{corollary}

\begin{proof}
Let $f\colon X\subseteq W\times K\to K$ be a definable family of functions. By Corollary \ref{cor:hasmac}, there is a definable partition of $X$ into sets $F\cup C\cup I$ where for all $w\in W$, $F_w$ is finite, $f_w|C_w$ is locally constant and $f_w|I_w$ is a local $C$-isomorphism. The result follows by applying Theorem \ref{thm:LJP} to the family $f|I$. 
\end{proof}

\bibliographystyle{plain}

\end{document}